\documentclass[reqno]{amsart} 
\usepackage{amsmath,amsthm,amsfonts,amssymb}
\usepackage{paralist}
\usepackage{graphicx}
\usepackage{verbatim}
\usepackage{color}
\usepackage{subcaption}

\theoremstyle{plain}
\newtheorem{theorem}{Theorem}[section]
\newtheorem{lemma}{Lemma}[section]

\theoremstyle{definition}
\newtheorem{definition}{Definition}[section]

\theoremstyle{remark}
\newtheorem{remark}{Remark}[section]

\numberwithin{equation}{section}

\title[Energy dissipation admissibility condition]
{Energy dissipation admissibility condition for conservation law systems
admitting singular solutions}
\author{Marko Nedeljkov and Sanja Ru\v{z}i\v{c}i\'{c}}
\address{Department of Mathematics and Informatics, University of Novi Sad,
Trg D.\ Obradovi\'{c}a 4, 21000 Novi Sad, Serbia}
\email{marko@dmi.uns.ac.rs, sanja.ruzicic@dmi.uns.ac.rs}
\keywords{admissibility of solutions, conservation law systems,
energy dissipation, shadow waves}

\begin{document}

\begin{abstract}
The main goal of the paper is to define and use a condition sufficient to
choose a unique solution to conservation law systems with a singular measure
in initial data.  Different approximations can lead 
to solutions with different distributional limits. 
The new notion called backward energy condition is then to single out a proper 
approximation of the distributional initial data. The definition is based
on the maximal energy dissipation defined in \cite{CD_1973}.
Suppose that a conservation law system admits a supplementary law
in space--time divergent form where the time component is a
(strictly or not) convex function. It could be an energy density or a mathematical 
entropy in gas dynamic models, for example. One of the admissibility
conditions is that a proper weak solution should maximally 
dissipate the energy or the mathematical entropy.
We show that it is consistent 
with other admissibility conditions in the case of Riemann problems for 
systems of isentropic gas dynamics with non-positive pressure
in the first part of the paper.
Singular solutions to these systems are described by 
shadow waves, nets of piecewise constant approximations with respect 
to the time variable. 
In the second part, we define and apply the backward energy condition 
for those systems when the initial data contains a delta measure
approximated by piecewise constant functions. 
\end{abstract}

\maketitle


\section{Introduction}

A conservation law system 
\[\partial_{t}U+\partial_{x}(F(U))=0,\; U:\Omega \to \mathbb{R}^{n}
\]
is called physical if there exists a pair of functions
$(\eta, Q)$, $\eta$ being convex, satisfying the additional
conservation law
\[\partial_{t}\eta(U)+\partial_{x}(Q(U))=0,
\]
for all classical smooth solutions $U$.
The function $\eta$ may be the energy density or 
the Lax (also called mathematical) entropy, for example.
In this paper, we are dealing with the isentropic systems of
gas dynamics and the function $\eta$ denotes the physical energy density.

The entropy admissibility condition for conservation laws 
based on some well-known physical systems
is introduced in \cite{CD_1973}:
The admissible weak solution to a physical system 
of conservation laws is the one that produces 
a maximal decrease of the quantity $\int \eta \, dt$. 
It is also called the principle of maximal energy dissipation
in the literature. We will call it the energy admissibility condition.

One can look in \cite{FGSW} for analysis of energy in  
compressible and incompressible isentropic Euler equations.
The above admissibility condition is not as usefull as
the usual ones in some cases.
For example, it cannot be used  for the Euler system of 
compressible gas when $\gamma < 5/3$, see \cite{Hsiao}.
Also, the authors in \cite{barbera,MW} 
found some examples when the use of that condition 
singles out physically incorrect solutions.
One can also see the results of the energy dissipation condition 
from a standpoint of relations between self-similar and oscillating solutions
constructed by the method of De Lellis and Sz\'{e}kelyhidi
for $n$-dimensional isentropic Euler system in \cite{F2014} 
and \cite{CK2014}.

In this paper, we will check if it is possible to 
use the energy admissibility condition to single out relevant solutions
of conservation law systems with shadow wave solutions introduced in \cite{mn2010}. 
They are used when the standard elementary waves
do not suffice for solving some Riemann problem. 
The delta function part annihilates 
a Rankine--Hugoniot deficit in the equations and 
the major concern is how to avoid an artificial deficit.
The usual methods used in the literature are the overcompressibility 
(all characteristics run into a shadow wave front) or 
the entropy condition (using convex entropy -- entropy flux pair).
We will compare the energy admissibility condition with those two.
The paper has two main parts. 

In the first part we deal with  three systems describing an isentropic flow of gas with
non-positive pressure. For each of them, 
shadow wave solution appear for some Riemann data and 
we apply the energy admissibility condition to them:

In the first two systems, modeling the pressureless and Chaplygin gas, 
the overcompressibility and the entropy conditions 
suffice to single out physically meaningful solutions.
That is a simple consequence of the fact that the energy is also
a mathematical entropy. 

In the third system, the generalized model of Chaplygin gas,
overcompressibility and the entropy condition were not enough
to singe out a unique solution (see \cite{MN_SR2017}). 
On the other hand, the energy condition successfully singles out a proper solution,
a combination of two shock waves instead of a single shadow wave.

In the second part of the paper, we will answer the following
question: How to choose an approximation of the initial data containing a 
combination of piecewise constant and delta function 
(called the delta Riemann data in the sequel)
and get a physically reasonable unique solution? 
The principal problem here is that two different
approximations of the same measure initial data give two different results in
distributional sense (approximate solutions not having 
the same distributional limit). 
There is also a practical reason for using that initial data: 
A procedure of solving a problem with a piecewise constant approximation 
of smooth initial data involves shadow wave interactions. 
The interaction problem then reduces to a special case of 
the above distributional initial data in a moment of interaction.

The idea is to define so-called backward energy condition.
We postulate that a proper choice of the initial data
approximation should produce a solution with minimal energy dissipation
in sufficiently small time interval. 
Note that the energy cannot rise and that a classical smooth 
solutions have zero energy dissipation. So, the ideal situation is to
choose the approximation such that a corresponding solution is smooth. 
When it is not possible, the postulate means that we choose the approximation
that produces a weak solution ``closest'' to a classical one. 

Let us note that there were incomplete attempts to solve pressureless
gas dynamics and related delta initial data problems by several
authors, but no one raised the question about the uniqueness
of a  solution and a meaning of the distributional initial data in
nonlinear systems (see (\ref{pgd_id})) up to our knowledge.
With the backward energy condition, we can single out
a global solution unique in the distributional 
sense for presureless gas dynamics and for the Chaplygin gas.
For the generalized Chaplygin model, we are also able to single out a proper initial data,
but a distributional limit of the solution will be known after one calculate
all possible wave interactions. That is a separate problem 
left for a further research.

\section{Admissible shadow wave solutions} 

Let
\begin{equation}\label{cl}
 U_{t}+F(U)_{x}=0
\end{equation} 
be a given conservation law system.
A weak solution $U$ is entropy admissible if 
\[\partial_t\eta(U)+\partial_x Q(U)\leq 0\]
holds in distributional sense for each convex entropy pair $(\eta, Q)$.
However, it is not always effective and there 
is also a question concerning its physical background for some systems. 
The energy admissibility condition states that the admissible solution 
is the one that dissipates the energy at the highest possible rate. 
In the recent years, a lot of systems having solutions (weak in some sense) 
with the delta function or its generalization are found. 
Shadow waves (see \cite{mn2010}) are used to solve such systems.
Our first goal is to check if the energy admissibility condition 
can be applied to such solutions.

The total energy of a solution $U$ in the interval $[-L,L]$ 
at time $t>0$ is given by 
\[ H_{[-L,L]}(U(\cdot,t)):=\int_{-L}^{L}\eta(U(x,t))dx.\]
A value $L>0$ is taken to be large enough to avoid a discussion 
about boundary conditions at least for some 
time $t<T$ significantly
greater than zero. To simplify the notation,
we write $H_{[-L,L]}(t)$ instead of $H_{[-L,L]}(U(\cdot,t))$.
Suppose that there is only one shadow wave 
\begin{equation} \label{sdw}
U^\varepsilon(x,t)=\begin{cases}U_{0}(x,t), & x<c(t)-\frac{\varepsilon}{2}t-x_\varepsilon \\
U_{0,\varepsilon}(t), & c(t)-\frac{\varepsilon}{2}t - x_{\varepsilon}< x < 
c(t) \\
U_{1,\varepsilon}(t), & c(t) < x < 
c(t)+\frac{\varepsilon}{2}t + x_{\varepsilon}\\
U_{1}(x,t), & x>c(t)+\frac{\varepsilon}{2}t+x_\varepsilon
\end{cases}
\end{equation}
passing through an interval $[T_{1},T_{2}] \times [-L,L]$.
The function $U_{i}(x,t)$ is smooth solutions
to the above system, the function $U_{i,\varepsilon}$ is 
called the intermediate state, $i=0,1$. The curve $x=c(t)$ is 
the front, $\xi(t)=\lim_{\varepsilon \to 0}(\frac{\varepsilon}{2} t
+x_{\varepsilon})(U_{0,\varepsilon}(t)+U_{1,\varepsilon}(t))$
is called the strength of the shadow wave.
Then
\[ 
\begin{split}
H_{[-L,L]}(t)= & \lim_{\varepsilon\to 0}
\int_{-L}^{L}\eta(U^\varepsilon(x,t))dx \\
= & 
\int_{-L}^{c(t)}\eta(U_{0}(x,t))dx
+\lim_{\varepsilon \to 0} \Big(\frac{\varepsilon}{2} t+x_\varepsilon\Big)\big(\eta(U_{0,\varepsilon}(t))
+\eta(U_{1,\varepsilon}(t))\big)\\
&+\int_{c(t)}^{L}\eta(U_{1}(x,t))dx
\end{split}
\]
for $t\in[T_1,T_2]$.
The energy production of (\ref{sdw}) at a time $t$ is 
\[
\begin{split}
& \frac{d}{dt}H_{[-L,L]}(t)= \lim_{\varepsilon\to 0}
\bigg(\!\Big(c'(t)-\frac{\varepsilon}{2}\Big)
\eta\Big(U_{0}\big(c(t)-\frac{\varepsilon}{2}t,t\big)\Big)
\\
&-\int_{-L}^{c(t)-\varepsilon t/2-x_\varepsilon}\!Q(U_{0}(x,t))_{x} dx 
 +\frac{\varepsilon}{2} \eta(U_{0,\varepsilon}(t))
+\frac{\varepsilon}{2} \eta(U_{1,\varepsilon}(t))\\
& +
 \int_{c(t)-\varepsilon t/2-x_\varepsilon}^{c(t)} 
\eta(U_{0,\varepsilon}(t))_{t} dx +\int_{c(t)}^{c(t)+\varepsilon t/2+x_\varepsilon} 
\eta(U_{1,\varepsilon}(t))_{t} dx \\
&-\Big(c'(t)+\frac{\varepsilon}{2}\Big)
\eta\Big(U_{1}\big(c(t)+\frac{\varepsilon}{2}t+x_\varepsilon,t\big)\Big)
-\int_{c(t)+\varepsilon t/2+x_\varepsilon}^{L} Q(U_{1}(x,t))_{x} dx\bigg),
\end{split}
\]
where we have used that $U_{\varepsilon}$ depends only on $t$,
the above limit exists and $\eta_{t}=-Q_{x}$ for smooth solutions. 
Finally, 
\[
\begin{split}
\frac{d}{dt}H_{[-L,L]}(t) = & -c'(t)
\big( \eta(U_{1}(c(t),t)) - \eta(U_{0}(c(t),t))\big)
+ Q(U_{1}(c(t),t)) \\ 
& - Q(U_{0}(c(t),t)) 
 +Q(U_{0}(-L,t)) - Q(U_{1}(L,t))\\
 &+\lim_{\varepsilon\to 0}\Big(\big(\frac{\varepsilon}{2} t +x_\varepsilon \big)
\frac{d}{dt}\big( \eta(U_{0,\varepsilon}(t))+\eta(U_{1,\varepsilon}(t))\big) \\
&
+ \frac{\varepsilon}{2}\big( \eta(U_{0,\varepsilon}(t))
+\eta(U_{1,\varepsilon}(t))\big)\Big). 
\end{split}
\]
Denote by 
\[\begin{split}
\mathcal{D}(t):= &-c'(t)
\big( \eta(U_{1}(c(t),t)) - \eta(U_{0}(c(t),t))\big)
+ Q(U_{1}(c(t),t)) 
 - Q(U_{0}(c(t),t)) \\
 &+\lim_{\varepsilon\to 0}\frac{d}{dt}\Big(\big(\frac{\varepsilon}{2} 
t+x_\varepsilon\big) \big( \eta(U_{0,\varepsilon}(t))+\eta(U_{1,\varepsilon}(t))\big)\Big)\\
 =& \frac{d}{dt}H_{[-L,L]}(t)-\big(Q(U_{0}(-L,t))-Q(U_{1}(L,t))\big)
\end{split}\]
the local energy production of shadow wave (\ref{sdw}) at time $t$. 
The energy production for a shock wave 
\[ 
U(x,t)=\begin{cases}U_{0}(x,t), & x<c(t) \\
U_{1}(x,t), & x>c(t)
\end{cases}
\]
is 
\[
\begin{split}
\frac{d}{dt}H_{[-L,L]}(t) = & -c'(t)
\big( \eta(U_{1}(c(t),t)) - \eta(U_{0}(c(t),t))\big) \\
& +\big( Q(U_{1}(c(t),t)) - Q(U_{0}(c(t),t)) \big) 
 +Q(U_{0}(-L,t)) - Q(U_{1}(L,t)) \\
 = & \mathcal{D}(t)+Q(U_{0}(-L,t)) - Q(U_{1}(L,t)). 
\end{split}
\]
It equals 
$\frac{d}{dt}H_{[-L,L]}(t)=Q(U_{0}(-L,t)) - Q(U_{1}(L,t))$
for a rarefaction wave (see \cite{CD_2012}), i.e.\ the local 
energy production $\mathcal{D}(t)$ equals zero due to a continuity 
of a rarefaction wave.

\begin{lemma}[\cite{mn2010}]
Denote by $(\eta,Q)$ a convex entropy pair for system (\ref{cl}).
Shadow wave solution (\ref{sdw}) satisfies
the entropy inequality $\partial_t \eta+\partial_x Q\leq 0$ in the sense of
distributions if 
\begin{equation} \label{E}
\begin{split}
\mathcal{D}(t)=-c'(t)[\eta]+[Q]+\lim_{\varepsilon \to
0}\frac{d}{dt}\Big(\big(\frac{\varepsilon}{2} t +x_\varepsilon\big)
\big( \eta(U_{0,\varepsilon}(t))+\eta(U_{1,\varepsilon}(t))\big)\Big) 
&\leq 0\\
\lim_{\varepsilon\to 0}\big(\frac{\varepsilon}{2} t +x_\varepsilon\big)
\Big(c'(t)\big(\eta(U_{0,\varepsilon}(t)+\eta(U_{1,\varepsilon}(t))\big)
-\big(Q(U_{0,\varepsilon}(t))+Q(U_{1,\varepsilon}(t))\big)\Big)&=0,
\end{split}
\end{equation}
where $[\eta]:=\eta(U_1(c(t),t))-\eta(U_0(c(t),t))$
and $[Q]:=Q(U_1(c(t),t))-Q(U_0(c(t),t))$.

\end{lemma}
\begin{remark}
For gas dynamics systems endowed with entropy pair $(\eta
,Q)$, $Q=u\eta$, the second condition in (\ref{E}) 
is automatically satisfied for shadow waves. That follows 
from the law of mass conservation. 
\end{remark}
\begin{lemma}\label{lemma:affine}
Denote by $(\eta,Q)$ a convex energy pair for system (\ref{cl}).
Suppose that shadow wave (\ref{sdw}) satisfies condition (\ref{E})
and $F(U_{0})\neq F(U_{1})$. Then, condition (\ref{E}) is also satisfied 
and $\mathcal{D}(t)$ is invariant under the transformation 
\begin{equation}\label{affine_eta}
\bar{\eta}(U)=\eta(U) +\sum_{j=1}^n a_{j} u^{j},\; 
\bar{Q}(U)=Q(U) +\sum_{j=1}^n a_{j} f^{j}(U)+\bar{c},
\end{equation} 
$U=(u^{1},\ldots,u^{n})$, $F(U)=(f^{1}(U),\ldots, f^{n}(U))$,
for every $a=(a_{1},\ldots,a_{n})\in\mathbb{R}^n$ and $\bar{c}\in\mathbb{R}$. 
\end{lemma}

Note that the pair $(\bar{\eta},\bar{Q})$
is also a convex energy pair for system (\ref{cl})
since we have added the affine term to $\eta$. 
\begin{proof}
One can easily prove that $D\bar{Q}=D\bar{\eta}DF$ and 
$D^2\eta=D^2\bar{\eta}.$ To prove 
\[\langle\partial_t \bar{\eta}(U^\varepsilon)
+\partial_x \bar{Q}(U^\varepsilon),\varphi\rangle
\approx\langle\partial_t {\eta}(U^\varepsilon)+\partial_x {Q}(U^\varepsilon),
\varphi\rangle, \; \varepsilon\to 0,\]
for any test function $\varphi\in C_{0}^{\infty}$, we use the results from 
\cite{mn2010}. We have
\[\begin{split}
\Lambda_{1}(t) &:=-c'(t)[U]+[F(U)]+\frac{d}{dt}\Big(\big(\frac{\varepsilon}{2}t+x_\varepsilon\big)
\big(U_{0,\varepsilon}(t)+U_{1,\varepsilon}(t)\big)\Big)
= \mathcal{O}(\varepsilon)\\
\Lambda_{2}(t) &:=\Big(\frac{\varepsilon}{2}t+x_\varepsilon\Big)\Big(c'(t)
\big(U_{0,\varepsilon}(t)+U_{1,\varepsilon}(t)\big)
-\big(F(U_{0,\varepsilon}(t))
+F(U_{1,\varepsilon}(t))\big)\Big) =\mathcal{O}(\varepsilon)
\end{split}\]
since (\ref{sdw}) is approximate solution to (\ref{cl}).
Using the procedure from the proof of Lemma 10.1 in \cite{mn2010} we get
\[\begin{split}
&\langle\partial_t \bar{\eta}(U^\varepsilon)+\partial_x
\bar{Q}(U^\varepsilon),\varphi\rangle=\langle\partial_t
{\eta}(U^\varepsilon)+\partial_x {Q}(U^\varepsilon),\varphi\rangle\\
& + \underbrace{\sum_{j=1}^{n} a_{j}\Big(\int_{0}^{\infty}
\Lambda_{1}^{j}(t)\varphi(c(t),t)\,dt + \int_{0}^{\infty}
\Lambda_{2}^{j}(t)\partial_x\varphi(c(t),t)\,dt
\Big)}_{=\mathcal{O}(\varepsilon)}+\mathcal{O}(\varepsilon),
\end{split}\]
where $\Lambda_{k}(t)=\big(\Lambda_{k}^{1}(t),\ldots,\Lambda_{k}^{n}(t)\big)$,
$k=1,2.$
\end{proof}
The lemma holds for classical weak solutions, too, as proved in 
\cite{CD_1973}. Its consequence is that 
that one can choose constants $a_i,$ $i=1,\ldots,n$ and $\bar{c}$ in such a way that 
\begin{equation}\label{no-flow}
Q(U_0(- L,t))=0,\; Q(U_1( L,t))=0.
\end{equation}
Then, there is no energy flow on the boundaries and the total energy decreases
\[\frac{d}{dt}\bar{H}_{[-L,L]}(t)\leq 0.\]
($\bar{H}_{[-L,L]}(t)$ denotes the total energy corresponding to
$\bar{\eta}.$)
So, for an energy pair $(\eta,Q)$, we can take affine transformation
(\ref{affine_eta}) such that (\ref{no-flow}) holds and the total 
energy decreases. Thus, the following definition is invariant under 
a change of energy flux at $x=-L$ and $x=L$ for shadow wave solutions also.
\begin{definition} (\cite{CD_1973})
The solution $U(x,t)$ to the system (\ref{cl}) satisfies energy 
admissibility condition if it minimizes
\[\min \Big\{\frac{d\bar{H}_{[-L,L]}(U(\cdot,t))}{dt}\Big|_{t=t_0+0}:\; U \text{ is
weak solution in } (t_{0},t)\Big\}.\]
\end{definition}

\subsection{Pressureless gas dynamics system} 

Take an isentropic, inviscid and compressible flow
of a gas in the absence of internal and external forces.
Its dynamics is described by mass and linear momentum conservation laws
\begin{equation}\label{pgd}
\begin{split}
\partial_t \rho + \partial_x(\rho u) & = 0\\
\partial_t (\rho u) + \partial_x(\rho u^2) & = 0.
\end{split}
\end{equation}
Here $\rho\geq 0$ and $u$ denote density and velocity of a fluid, respectively.
The model can be obtained from the compressible Euler equations of gas
dynamics by letting the pressure tend to zero. It is used to describe 
a formation of large structures in the Universe and behavior of particles that
stick under collision (\cite{b_g, WRS, Zeldovich}).

System (\ref{pgd}) is weakly hyperbolic with both characteristic fields
being linearly degenerate and characteristic speeds equal to velocity,
$\lambda_1(\rho,u)=\lambda_2(\rho,u)=u$.

The solution to system (\ref{pgd}) with the Riemann initial data
\begin{equation}\label{pgd_ri}
(\rho, u)(x,0)=\begin{cases}
(\rho_0,u_0), & x<0\\
(\rho_1,u_1), & x>0
\end{cases}
\end{equation} 
consists of two contact discontinuities connected by the vacuum state
\begin{equation*}
U(x,t)=\begin{cases}
(\rho_0,u_0), & x<u_0 t\\
(0,\frac{x}{t}), & u_0 t<x<u_{1}t\\
(\rho_{1},u_{1}), & x>u_{1} t
\end{cases}
\end{equation*}
if $u_0\leq u_1$. This solution will be called CD wave combination.
Otherwise, there exists a singular solution given in the form of a delta
shock approximated by a shadow wave. 
It is given by
\begin{equation}\label{wSDW_l,r}
U^\varepsilon(x,t)=\begin{cases}
(\rho_0,u_0),	&
x<c-\frac{\varepsilon}{2}t\\
(\rho_\varepsilon,u_s), &
c-\frac{\varepsilon}{2}t<x<c+\frac{\varepsilon}{2}t \\
(\rho_1,u_1), & x>c+\frac{\varepsilon}{2}t.
\end{cases}
\end{equation}
It is overcompressive,  $u_1\leq u_s(0)\leq u_0$,
$u_{s}=y:=\frac{\sqrt{\rho_0}u_0+\sqrt{\rho_1}u_1}{\sqrt{\rho_0}+\sqrt{\rho_1}}$,
and its strength equals $\xi=\sqrt{\rho_0\rho_1}(u_0-u_1)t$ (only for $\rho$-variable,
the velocity is bounded).

\subsection{Chaplygin gas}
The gas dynamics model with negative pressure
\begin{equation}\label{chap}
\begin{split}
\partial_{t}\rho + \partial_{x} (\rho u) & =0 \\
\partial_{t} (\rho u)+ \partial_{x} \Big( \rho
u^{2}-\frac{1}{\rho}\Big) & =0
\end{split}
\end{equation}
is introduced to describe the aerodynamics force acting on 
a wing of an airplane in \cite{Chaplygin}. In recent years, it is found that
it can be used in cosmology as a model of the dark energy in the Universe,
see \cite{ph_dark}.
Here $u$ and $\rho$ denote the velocity and the density of a fluid,
respectively.

System (\ref{chap}) is strictly hyperbolic having two distinct eigenvalues
$\lambda_1(\rho,u)=u-\frac{1}{\rho}<\lambda_2(\rho,u)=u+\frac{1}{\rho}$ 
with both fields being linearly degenerate. 
If $\lambda_1(\rho_0,u_0)<\lambda_2(\rho_1,u_1)$, there exists classical
solution to the Riemann problem (\ref{chap}, \ref{pgd_ri}).
It consists of two contact discontinuities connected by the
constant intermediate state
$(\rho_m,u_m)=\big(\frac{2}{\lambda_2(\rho_1,u_1)-\lambda_1(\rho_0,u_0)},
\frac{\lambda_2(\rho_1,u_1)+\lambda_1(\rho_0,u_0)}{2}\big)$ (see
\cite{mn2014}). Otherwise, the solution is overcompressive
shadow wave  propagating with the constant speed
\begin{equation}\label{speed_chap}
s=\frac{[\rho u]}{[\rho]}+\frac{1}{[\rho]}\sqrt{[\rho u]^2-[\rho][\rho
u^2-\rho^{-1}]}=\frac{[\rho u]}{[\rho]}+\frac{\kappa}{[\rho]},
\end{equation}
where $[\cdot]$ is a jump across a discontinuity,
and $\kappa$ is the Rankine-Hugoniot deficit, $\kappa:=s[\rho]-[\rho u]$. 
It is worthwhile to mention that
the term under the square root in (\ref{speed_chap}) can be written as 
\[\rho_{0}\rho_{1}\big(\lambda_{1}(\rho_{0},
u_{0})-\lambda_{1}(\rho_{1},u_{1})\big)\big(\lambda_{2}(\rho_{0},
u_{0})-\lambda_{2}(\rho_{1},u_{1})\big).\]
So, the shadow wave solution exists only if
$\lambda_{1}(\rho_{0},u_{0})\geq \lambda_{2}(\rho_{1},u_{1})$,
i.e.\ if there is no classical weak solution.
Also, the overcompressibility and the entropy condition 
for any entropy pair are equivalent. The most interesting
choice is $\eta$ being the physical energy density
$\eta=\rho u^2+\frac{1}{\rho}$ with corresponding flux
$Q=\rho u^3-\frac{u}{\rho}$ because this choice leads directly to 
the fact that the energy admissibility condition is satisfied, too. 
A solution to Riemann problem for Chaplygin gas system is weakly unique 
(the distributional limit is unique).

\subsection{Generalized Chaplygin gas model}

In this section, we will see a real use of the energy admissibility condition 
for unbounded solutions.
It suffices to single out a proper (physically relevant) solution contrary to 
other admissibility criteria.

The generalized Chaplygin gas model consists of
the mass and momentum conservation laws
\begin{equation}\label{genchap}
\begin{split}
\partial_{t}\rho + \partial_{x} (\rho u) & =0 \\
\partial_{t} (\rho u)+ \partial_{x} \Big( \rho
u^{2}-\frac{1}{\rho^{\alpha}}\Big) & =0,
\end{split}
\end{equation}
where $0<\alpha<1$. That is a strictly hyperbolic system with eigenvalues
$\lambda_1(\rho,u)=u-\sqrt{\alpha}\rho^{-\frac{1+\alpha}{\alpha}}
<\lambda_2(\rho,u)=u+\sqrt{\alpha}\rho^{-\frac{1+\alpha}{\alpha}}$.
A classical weak solution to the Riemann
problem (\ref{genchap}, \ref{pgd_ri}) exists
if the right initial state $(\rho_1,u_1)$
is above the curve
\[
\Gamma_{ss}(\rho_{0},u_{0}):\;u=u_0- \rho_{0}^{-\frac{1+\alpha}{2}}-
\rho^{-\frac{1+\alpha}{2}}.
\]
That solution is a combination of elementary waves (shock and rarefaction
waves) since both characteristic fields are genuinely nonlinear. 
For $(\rho_{1},u_{1})$ below and at 
the curve $\Gamma_{ss}(\rho_{0},u_{0})$ 
a unique solution is the shadow wave as proved in \cite{MN_SR2017}.
\begin{lemma}
There exists a shadow wave solution (\ref{sdw}) with $U=(\rho,u)$
to (\ref{genchap}, \ref{pgd_ri}) if
\[\rho_{0}\rho_{1}(u_{0}-u_{1})^2>(\rho_{0}-\rho_{1})
(\rho_{1}^{-\alpha}-\rho_{0}^{-\alpha}).\]
Its speed is given by
\begin{equation*}
c=\begin{cases}
\frac{u_1\rho_1-u_0\rho_0+\kappa_{1}}{\rho_1-\rho_0},
& \text{if } \rho_{0}\neq
\rho_{1}\\
\frac{1}{2}(u_{0}+u_{1}), & \text{if } \rho_{0}= \rho_{1},
\end{cases}
\end{equation*}
where
$\kappa_1=\lim_{\varepsilon\to 0} \varepsilon
\rho_\varepsilon=\sqrt{\rho_0\rho_1(u_0-u_1)^2
-(\rho_0-\rho_1)(\rho_1^{-\alpha}-\rho_0^{-\alpha})}$
denotes Ran\-ki\-ne-Hugoniot deficit and $\delta$ is supported by the line
$x=c(t)=ct$.
\end{lemma}

In most of the cases from the literature, it suffices to
use overcompressibility condition to exclude non-wanted singular 
solutions containing the delta function.
However, that is not the case here:
There is an area above $\Gamma_{ss}$ where both
overcompressive shadow wave and classical two shock solution exist.
The first shock propagates with speed $c_{1}=u_0-A(\rho_{0},\rho_{m})$
and joins $(\rho_{0},u_{0})$ on the left to $(\rho_{m},u_{m})$ on the right,
while the second one propagates with speed
$c_{2}=u_{m}+A(\rho_{m},\rho_{1})=u_{1}+A(\rho_{1},\rho_{m})$
and joins $(\rho_{m},u_{m})$ to $(\rho_{1},u_{1})$. 
Here,
\begin{equation}\label{A_0m}
A(\rho_{i},\rho_{j}):=\sqrt{\frac{\rho_{j}}{\rho_{i}}
\frac{\rho_{i}^{-\alpha}-\rho_{j}^{-\alpha}}{\rho_{j}-\rho_{i}}}.
\end{equation}
and intermediate state $(\rho_m,u_m)$ is determined by the following relations
\begin{equation}\label{intermediate}
\begin{split}
u_m & =u_0-\sqrt{\frac{\rho_m-\rho_0}{\rho_0\rho_m}
(\rho_0^{-\alpha}-\rho_m^{-\alpha})},
\quad \rho_m>\rho_0 \\
u_m &=u_1+\sqrt{\frac{\rho_m-\rho_1}{\rho_1\rho_m}
(\rho_1^{-\alpha}-\rho_m^{-\alpha})},
\quad \rho_m>\rho_1.
\end{split}
\end{equation}

Thus, one has to find an additional admissibility criterion which will
exclude the overcompressive shadow wave solution 
in the area above $\Gamma_{ss}$ curve. 
The system (\ref{genchap}) possesses infinitely many convex entropies, so
one can use entropy condition given in \cite{mn2010}.
The shadow wave solution (\ref{wSDW_l,r}) with
a constant speed $c$ is admissible if 
\begin{equation}\label{entr}
\begin{split}
\mathcal{D}:= 
 -c[\eta]+[Q]+\lim_{\varepsilon\to 0}\varepsilon
\eta(\rho_\varepsilon, u_\varepsilon) & \leq 0\\
 \lim_{\varepsilon\to 0} (- \varepsilon c \eta(\rho_\varepsilon,
u_\varepsilon)+\varepsilon Q(\rho_\varepsilon, u_\varepsilon)) & =0
\end{split}
\end{equation}
for each convex entropy pair $(\eta, Q)$. 
The idea was implemented in \cite{MN_SR2017}, but without a complete success
probably due to the lack of precise approximations of the modified Bessel
functions of the second kind. However, all numerical simulations have shown
that shadow wave solution is not admissible above $\Gamma_{ss}$. 
Here, we are about to show that energy admissibility condition 
successfully excludes the unwanted shadow wave solution in that area using 
the following energy--energy flux pair:
\begin{equation}\label{en_pair}
\eta(\rho,u)=\frac{1}{2}\rho u^2+\frac{1}{1+\alpha}\rho^{-\alpha} \text{ and }
Q(\rho,u)=\frac{1}{2}\rho u^3-\frac{\alpha}{1+\alpha}\rho^{-\alpha}u.
\end{equation}

\subsubsection{The energy dissipation}

The energy admissibility condition will exclude the unwanted
shadow wave solution above $\Gamma_{ss}$ if its energy production is
greater then energy production of classical $S_1+S_2$ solution. 
The local energy production of the two shock combination connecting states $(\rho_0,u_0)$ and $(\rho_1,u_1)$ is constant and
given by
\[\begin{split}
\mathcal{D}^{cl} :=&
-c_1(\eta(\rho_m,u_m)-\eta(\rho_0,u_0))+Q(\rho_m,u_m)-Q(\rho_0,u_0) \\
& -c_2(\eta(\rho_1,u_1)-\eta(\rho_m,u_m))+Q(\rho_1,u_1)-Q(\rho_m,u_m),
\end{split}\]
while
\begin{equation*}
\begin{split}
\mathcal{D}^{sdw} &:=-c[\eta]+[Q]+\lim_{\varepsilon\to 0}\varepsilon
\eta(\rho_\varepsilon, u_\varepsilon)= -c[\eta]+[Q]+\frac{1}{2}c^2\kappa_1
\end{split}
\end{equation*}
is the local energy production for the shadow wave.
The following two relations 
\begin{equation}\label{u_kappa}
\begin{split}
u_0-u_1
&=\frac{\rho_m-\rho_0}{\rho_m}A(\rho_{0},\rho_{m})
+\frac{\rho_m-\rho_1}{\rho_m}A(\rho_{1},\rho_{m}),\\
\kappa_1 &=\sqrt{(\rho_m-\rho_0)(\rho_m-\rho_1)}(c_2-c_1)
\end{split}
\end{equation}
are consequences of (\ref{intermediate}).
Suppose that $\rho_1\neq \rho_0.$ Then, (\ref{u_kappa}) implies
\begin{equation}\label{rel_c}
\begin{split}
c-c_1&=\frac{\rho_1}{\rho_1-\rho_0}(u_1-u_0)
+\frac{\kappa_1}{\rho_1-\rho_0}+A(\rho_{0},\rho_{m})\\
&=\frac{\rho_m-\rho_1}{\rho_1-\rho_0}
\Big(-1+\sqrt{\frac{\rho_m-\rho_0}{\rho_m-\rho_1}}\Big)(c_2-c_1)>0\\
c-c_2 &=
\frac{\rho_0}{\rho_1-\rho_0}(u_1-u_0)+\frac{\kappa_1}{\rho_1-\rho_0}
-A(\rho_{1},\rho_{m})\\
&=\frac{\rho_m-\rho_0}{\rho_1-\rho_0}
\Big(-1+\sqrt{\frac{\rho_m-\rho_1}{\rho_m-\rho_0}}\Big)(c_2-c_1)<0.
\end{split}
\end{equation}
We have used that
$\frac{\rho_m-\rho_0}{\rho_m-\rho_1}>1$ if $\rho_1>\rho_0$. The inequalities 
in (\ref{rel_c}) have the opposite signs if $\rho_1<\rho_0.$
That proves that the speed of shadow wave is between shock speeds, $c_1<c<c_2.$
If $\rho_0=\rho_1$, then
$c-c_1=c_2-c=\frac{\rho_{0}}{\rho_{m}}A(\rho_{0},\rho_{m}).$

\begin{theorem}
The solution to the Riemann problem (\ref{genchap},
\ref{pgd_ri}) is a combination of classical elementary waves if
\begin{equation*} 
u_1> u_0-\rho_{0}^{-\frac{1+\alpha}{2}}- \rho_{1}^{-\frac{1+\alpha}{2}},
\end{equation*}
or shadow wave (\ref{sdw}) otherwise.
Both solutions satisfy the energy admissibility condition.
\end{theorem}
\begin{proof}
If $(\rho_{1},u_{1})$ lies in the area where only one solution exists, 
there is nothing to prove. Suppose that there exist two solutions for 
$(\rho_{1},u_{1})$ above the curve $\Gamma_{ss}(\rho_{0},u_{0})$,
a classical two shock solution and a shadow wave.
To prove the theorem it thus suffices to show that $\mathcal{D}^{sdw}-\mathcal{D}^{cl}>0$
in that area. We have
\[
\mathcal{D}^{sdw}-\mathcal{D}^{cl}=\eta(\rho_{m},u_{m})(c_{1}-c_{2})
+\eta(\rho_{0},u_{0})(c-c_{1})+\eta(\rho_{1},u_{1})(c_{2}-c)
+\frac{1}{2}c^2\kappa_1.\]
Suppose that $\rho_{0}\neq \rho_{1}$. Using (\ref{rel_c}), we get
\[
\begin{split}
\frac{\mathcal{D}^{sdw}-\mathcal{D}^{cl}}{c_{2}-c_{1}}
=& -\eta(\rho_{m},u_{m})+\frac{\rho_m-\rho_1}{\rho_1-\rho_0}
\Big(-1+\sqrt{\frac{\rho_m-\rho_0}{\rho_m-\rho_1}}\Big)
\eta(\rho_{0},u_{0})\\
&+ \frac{\rho_m-\rho_0}{\rho_1-\rho_0}
\Big(1-\sqrt{\frac{\rho_m-\rho_1}{\rho_m-\rho_0}}\Big)\eta(\rho_{1},u_{1})
+ \frac{1}{2}c^2\sqrt{(\rho_m-\rho_0)(\rho_m-\rho_1)}\\
= & \frac{1}{1+\alpha}I_1+\frac{1}{2}I_2,
\end{split}
\]
where
\[
\begin{split}
I_1= &\rho_{1}^{-\alpha}-\rho_{m}^{-\alpha}-\frac{\rho_m-\rho_1}{\rho_1-\rho_0}
\Big(-1+\sqrt{\frac{\rho_m-\rho_0}{\rho_m-\rho_1}}\Big)
(\rho_{1}^{-\alpha}-\rho_{0}^{-\alpha}),\\
I_2= & -\rho_{m}u_{m}^2+\frac{\rho_m-\rho_1}{\rho_1-\rho_0}
\Big(-1+\sqrt{\frac{\rho_m-\rho_0}{\rho_m-\rho_1}}\Big)\rho_{0}u_{0}^2\\
&+ \frac{\rho_m-\rho_0}{\rho_1-\rho_0}
\Big(1-\sqrt{\frac{\rho_m-\rho_1}{\rho_m-\rho_0}}\Big)
\rho_{1}u_{1}^{2}+c^2\sqrt{(\rho_m-\rho_0)(\rho_m-\rho_1)}.
\end{split}
\]
Let us prove that $I_1>0$.
\[
\begin{split}
I_1=& (\rho_{1}^{-\alpha}-\rho_{m}^{-\alpha})
\frac{\rho_{m}-\rho_{0}}{\rho_{1}-\rho_{0}}
-(\rho_{0}^{-\alpha}-\rho_{m}^{-\alpha})\frac{\rho_{m}-\rho_{1}}{\rho_{1}
-\rho_{0}}\\
&+(\rho_{0}^{-\alpha}-\rho_{1}^{-\alpha})
\frac{\rho_{m}-\rho_{0}}{\rho_{1}-\rho_{0}}
\sqrt{\frac{\rho_{m}-\rho_{1}}{\rho_{m}-\rho_{0}}}\\
=& \frac{\rho_{m}-\rho_{0}}{\rho_{1}-\rho_{0}}(\rho_{m}-\rho_{1})
\bigg(\frac{\rho_{1}^{-\alpha}-\rho_{m}^{-\alpha}}{\rho_{m}-\rho_{1}}
\Big(1-\sqrt{\frac{\rho_{m}-\rho_{1}}{\rho_{m}-\rho_{0}}}\Big)\\
&+\frac{\rho_{0}^{-\alpha}-\rho_{m}^{-\alpha}}{\rho_{m}-\rho_{0}}
\Big(-1+\sqrt{\frac{\rho_{m}-\rho_{0}}{\rho_{m}-\rho_{1}}}\Big)\bigg).
\end{split}
\]
One can easily see that
\[a_{1}:=1-\sqrt{\frac{\rho_{m}-\rho_{1}}{\rho_{m}-\rho_{0}}}>0
\text{ and }
a_{2}:=-1+\sqrt{\frac{\rho_{m}-\rho_{0}}{\rho_{m}-\rho_{1}}}> 0\]
if $\rho_{1}>\rho_{0}$. Both $a_{1}$ and $a_{2}$ are negative if
$\rho_{1}<\rho_{0}$ and $I_{1}>0$ because
$\rho_{m}>\rho_{0}$ and $\rho_{m}>\rho_{1}$.

Eliminating 
$u_{0}$ and $u_{m}$ from $I_2$, we get
\[
\begin{split}
I_{2}= &\frac{\rho_{m}-\rho_{0}}{(\rho_{1}-\rho_{0})^2}
\frac{\rho_{m}-\rho_{1}}{\rho_{m}}\bigg(\frac{1}{2}
\big(\rho_{1}a_{2}-\rho_{0}a_{1}\big)
\Big(\sqrt{\rho_{0}^{-\alpha}-\rho_{m}^{-\alpha}}
-\sqrt{\rho_{1}^{-\alpha}-\rho_{m}^{-\alpha}}\Big)^2\\
&+ \sqrt{\big(\rho_{0}^{-\alpha}-\rho_{m}^{-\alpha}\big)
\big(\rho_{1}^{-\alpha}-\rho_{m}^{-\alpha}\big)}
\Big(\big(\rho_{1}a_{2}-\rho_{0}a_{1}\big)\\
&-\sqrt{\rho_{0}\rho_{1}}\frac{\big(\sqrt{\rho_{m}-\rho_{1}}
-\sqrt{\rho_{m}-\rho_{0}}\big)^2}{\sqrt{(\rho_{m}-\rho_{1})
(\rho_{m}-\rho_{0})}}\Big)\bigg).
\end{split}
\]
Then
\[\rho_{1}a_{2}-\rho_{0}a_{1}=\Big(1-\underbrace{\sqrt{\frac{\rho_{m}-\rho_{0}}
{\rho_{m}-\rho_{1}}}}_{=:a}\Big)\Big(\rho_{0}
\sqrt{\frac{\rho_{m}-\rho_{1}}{\rho_{m}-\rho_{0}}}-\rho_1\Big)>0.\]
The above inequality follows from the fact that $a>1$ for $\rho_{1}>\rho_{0}$.
That implies $(1-a)<0$ and $\frac{\rho_{0}}{a}-\rho_1<0$. 
The same holds for $\rho_{1}<\rho_{0}.$
Finally, 
\[\begin{split}
&\big(\rho_{1}a_{2}-\rho_{0}a_{1}\big)-\sqrt{\rho_{0}\rho_{1}}
\frac{\big(\sqrt{\rho_{m}-\rho_{1}}-\sqrt{\rho_{m}-\rho_{0}}\big)^2}
{\sqrt{(\rho_{m}-\rho_{1})(\rho_{m}-\rho_{0})}}\\
=&-\big(\sqrt{\rho_{0}}-\sqrt{\rho_{1}}\big)^{2}
+\big(\sqrt{\rho_{0}}-\sqrt{\rho_{1}}\big)\Big(\sqrt{\rho_{0}}
\frac{1}{a}-\sqrt{\rho_{1}}a\Big)>0
\end{split}\]
since
\[\frac{\sqrt{\rho_{0}}\frac{1}{a}-\sqrt{\rho_{1}}a}{\sqrt{\rho_{0}}
-\sqrt{\rho_{1}}}>1.\]
Thus, $I_2>0$ and $\mathcal{D}^{sdw}-\mathcal{D}^{cl}>0$.\\
The case $\rho_0=\rho_1$ is simpler,
\[\mathcal{D}^{sdw}-\mathcal{D}^{cl}
=(c_2-c_1)(\rho_{0}^{-\alpha}-\rho_{m}^{-\alpha})
\Big(\frac{1}{2}\frac{\rho_{m}-\rho_{0}}{\rho_{m}}
+\frac{1}{1+\alpha}\Big)>0.\]
\end{proof}

\begin{remark}
Note that this result coincides with the one from \cite{MNtam} obtained 
using the so-called vanishing pressure method.
\end{remark}

\section{Backward energy condition}

It is known that in the case of hyperbolic systems singularities
naturally develop even if initial data are smooth. 
Now, suppose that we have singular initial data. 
The idea behind the 
new admissibility criteria is to single out a maximally regular 
weak solution with such data. We believe that such a solution 
is the best approximation of a real-world process. Let us give a 
simple illustration of the idea. 
Consider the inviscid Burgers' equation with arbitrary initial data 
and suppose that a solution to such problem is a step function at fixed time 
$t=T$, meaning that it is given in the form of shock wave at least for $t>T$. 
In general, that happens even if the initial data is smooth. The value $T$ 
may be the point where a classical smooth solution 
if it exists, breaks down, i.e.\ where its gradient explodes. 
Starting from that moment, we have to use another 
more robust solution definition and apply new methods 
mostly based on some kind of approximation that will give 
a complete solution to the initial data problem. 
If the initial data is discontinuous, 
then the use of such methods is inevitable. In order to avoid the 
possibility of choosing some other weak over classical smooth solution
as long as it solution exists, the backward energy condition
has to be defined in such a way that a smooth solution is favoured. 
The philosophy behind that idea is that smooth solutions are the most natural 
ones. By using that kind of reasoning, we will choose 
the approximation of the initial that gives a unique weak solution 
(distributional limit) with a minimal energy dissipation.

The main example will be the system of gas dynamics
\begin{equation}\label{s}
\begin{split}
\partial_t \rho + \partial_x(\rho u) & = 0\\
\partial_t (\rho u) + \partial_x(\rho u^2+p(\rho)) & = 0,
\end{split}
\end{equation}
with pressure functions $p(\rho) \leq 0$ satisfying 
$p'(\rho)\geq 0$ and $\rho p''(\rho)+2p'(\rho)\geq 0$.
Due to the lack of a positive pressure we expect to have
solutions with mass being infinite. 
Thus, take the measure (or distributional) initial data
\begin{equation}\label{pgd_id}
(\rho, u)(x,0)=\begin{cases}
(\rho_0,u_0), & x<0\\
(\rho_1,u_1), & x>0
\end{cases}+ (\xi_\delta,0)\delta,
\end{equation}
where $\rho_0,\rho_1>0$, $\xi_\delta>0$.

Like in the first part, we analyse three different types of pressure:
\begin{itemize}
\item 
$p(\rho)\equiv 0$ (pressureless gas dynamics), 
\item 
$p(\rho)=-\rho^{-1}$ (Chaplygin model) and
\item
$p(\rho)=-\rho^{-\alpha}$, $\alpha\in (0,1)$ (generalized Chaplygin model). 
\end{itemize}

In all cases, there exists a unique solution to Riemann problem
being the shadow wave or a combination of elementary waves. 

We will use wave tracking approach from \cite{mn_sr2019}
to solve initial data problem (\ref{s}, \ref{pgd_id}).
The initial data are approximated by the piecewise constant function
\begin{equation}\label{id_approx}
U_{\mu}(x,0):=(\rho_{\mu},u_{\mu})(x,0)=\begin{cases}
(\rho_0,u_0), & x<-\frac{\mu}{2}\\
(\frac{\xi_\delta}{\mu},u_\delta), & -\frac{\mu}{2}<x<\frac{\mu}{2}\\
(\rho_1,u_1), & x>\frac{\mu}{2}
\end{cases}
\end{equation}
depending on the parameter $\mu \gg \varepsilon$. 
The parameter 
$1 \gg \varepsilon>0$ is used in a construction of shadow waves (\ref{sdw})
used for solving (\ref{s}, \ref{id_approx}) together with 
the other elementary waves. 
The parameter $\mu$ tends to zero but significantly slower than $\varepsilon$. 
A value $u_\delta$ is artificially added into (\ref{id_approx})
and it does not have any influence on a distributional value of the
initial data. A way to single out a proper solution is to 
find a value of $u_{\delta}$ that eliminates all unphysical
solutions. A distributional physically meaningful
solution to the problem (\ref{s}, \ref{pgd_id}) is then obtained by letting
$\mu\to 0$ (that makes $\varepsilon \to 0$ even faster).
We are looking for initial data 
that gives a minimal energy dissipation (or equivalently, a maximal 
local energy production) of a unique solution 
corresponding to initial data (\ref{id_approx}).

\begin{definition}[Backward energy condition] \label{bec}
A solution $U$ satisfies the backward energy condition if it is an admissible
solution to (\ref{s}, \ref{pgd_id}) for some $u_\delta$ such that
\begin{equation} \label{a}
\frac{d}{dt}H_{[-L,L]}(\tilde{U}(\cdot,0+))\geq 
\frac{d}{dt}H_{[-L,L]}(U(\cdot,0+))\, \text{ for } \mu \text{ small enough},
\end{equation}
where $\tilde{U}$ is an admissible solution corresponding to some other
initial value $\tilde{u}_{\delta}$.
If there is more than one solution that satisfies the above condition,
we choose the one that corresponds to initial data 
with a minimal total initial energy. 
Here, $U$ and $\tilde{U}$ are approximated solutions depending on 
$\varepsilon\ll \mu$.
\end{definition}

\begin{remark}
Note that a choice of an energy pair $\tilde{\eta}$ and $\tilde{Q}$ from 
Lemma \ref{lemma:affine} has a direct impact on the 
second criterion in Definition \ref{bec}. The first one, (\ref{a}),
is independent of the choice.
\end{remark}

The general algorithm goes as follows.
Assume that a small parameter $\mu>0$ and a value of
$u_\delta$ from (\ref{id_approx}) are given.
Initially, an approximate solution is a solution to double Riemann
problem (\ref{id_approx}). Those waves are uniquely determined by the
relationship between $u_0$, $u_\delta$ and $u_1$ as we saw in
the first part. The first wave or wave combination 
emanates from $\big(-\frac{\mu}{2},0\big)$
and connects the state $U_{0}=(\rho_0,u_0)$ with
$U_{\delta}=\big(\frac{\xi_\delta}{\mu},u_\delta\big)$, 
while the second one emanates from $\big(\frac{\mu}{2},0\big)$
and connects the state $U_{\delta}$ with $U_{1}=(\rho_1,u_1)$.
Then, continue by following further wave interactions. 
The procedure resembles the method of higher order shadow waves 
introduced in \cite{mn2014}. 

Now, we shall apply the backward energy condition to three 
different systems of the form (\ref{s}) described above.

\subsection{Pressureless gas dynamics system} 

\subsubsection{Weak solution and the backward energy condition} 
Recall, the energy--energy flux pair for the system is
$\eta=\frac{1}{2}\rho u^2$ and $Q=\eta u$.
The local energy production of a shadow wave at time $t$ is 
\[\begin{split}
\mathcal{D}(t)&=-u_s(t)[\eta]+[Q]+\lim_{\varepsilon \to
0}\frac{d}{dt}\Big(2\big(\tfrac{\varepsilon}{2} t+x_\varepsilon\big)
\eta\big(U_\varepsilon(t)\big)\Big)\\
&= -\frac{1}{2}\big(\rho_0(u_0-u_s(t))^3+\rho_1(u_s(t)-u_1)^3\big).
\end{split}\]
It is negative for a shadow wave and zero for a contact discontinuity.
These facts will be used to choose a proper initial data.

\noindent
{\it Case $A_1$} ($u_\delta<u_0<u_1$). The solution is 
given by the shadow wave connecting $U_{0}$ and $U_{\delta}$ 
and the CD wave combination connecting $U_{\delta}$ and $U_{1}$.
The local energy production of the solution fully comes from the 
shadow wave,
\[\begin{split}
\mathcal{D}& =-\frac{1}{2}\rho_0(u_0-y_{0,\delta})^3
-\frac{1}{2}\frac{\xi_\delta}{\mu}
(y_{0,\delta}-u_\delta)^3
 =-\frac{1}{2}\frac{\rho_0\xi_{\delta}/\mu}{(\sqrt{\xi_
{\delta}/\mu}+\sqrt{\rho_0})^2}(u_0-u_\delta)^3<0
\end{split}\]
for $t$ small enough.
It is clear that $\mathcal{D}<0$ and increases to zero as $u_{\delta}\to
u_{0}$. 

\noindent
{\it Case $A_2$} ($u_0\leq u_\delta \leq u_1$). In this case, the energy is
conserved and we have $\mathcal{D}=0$
for each $t\geq 0$ and each $u_\delta\in [u_0,u_1].$
The total initial energy has a minimum at
\begin{equation} \label{iemin}
u_{\delta}= \begin{cases} 0, & \text{ if } \mathop{\rm sign}(u_{0} u_{1})<0 \\ 
\min\{u_{0}, u_{1}\}, & \text{ if } \mathop{\rm sign}(u_{0}), 
\mathop{\rm sign}(u_{1})\geq 0 \\
\max\{u_{0}, u_{1}\}, & \text{ if } \mathop{\rm sign}(u_{0}), 
\mathop{\rm sign}(u_{1})\leq 0. \end{cases} 
\end{equation}

\noindent
{\it Case $A_3$} ($u_0<u_1<u_\delta$).
We have
$\mathcal{D}<0$ and $\mathcal{D}\to 0$ as $u_\delta\to u_1$ since
$\mathcal{D}=-\frac{1}{2}\frac{\rho_{1}\xi_{\delta}/\mu(u_{\delta}-u_{1})^3}
{(\sqrt{\xi_{\delta}/\mu}+\sqrt{\rho_{1}})^2}$ for $t$ small enough
(like in Case $A_1$).
\medskip

\noindent
{\it Case $B_1$} ($u_\delta<u_1<u_0$). 
Now,
\[\mathcal{D}=-\frac{1}{2}\rho_0\frac{\xi_{\delta}/\mu}{(\sqrt{\xi_
{\delta}/\mu}+\sqrt{\rho_0})^2}(u_0-u_\delta)^3
<-\frac{1}{2}\rho_{0}(u_0-u_{1})^3\]
for $t$ small enough.

\noindent
{\it Case $B_{2}$} ($u_1\leq u_\delta \leq u_0$ with at least one inequality 
being strict). In this case
\[
\begin{split}
\mathcal{D}&= -\frac{1}{2}\frac{\rho_0\xi_{\delta}/\mu}
{(\sqrt{\xi_{\delta}/\mu}+\sqrt{\rho_0})^2}(u_0-u_\delta)^3
-\frac{1}{2}\frac{\rho_{1}\xi_{\delta}/\mu}{(\sqrt{\xi_{\delta}/\mu}
+\sqrt{\rho_{1}})^2}(u_{\delta}-u_{1})^3
 \end{split}
\]
for $t$ small enough. $\mathcal{D}$ is a cubic function in $u_\delta$ having a
maximum in 
\[
y^{\mu}=\frac{u_0\sqrt{r_{0}^{\mu}}+u_1\sqrt{r_{1}^{\mu}}}
{\sqrt{r_{0}^{\mu}}+\sqrt{r_{1}^{\mu}}},\;\text{where } r_{0}^{\mu}
=\frac{\rho_{0}\xi_{\delta}/\mu}{(\sqrt{\xi_{\delta}/\mu}+\sqrt{\rho_{0}})^2}
\text{ and } r_{1}^{\mu}=\frac{\rho_{1}\xi_{\delta}/\mu}
{(\sqrt{\xi_{\delta}/\mu}+\sqrt{\rho_{1}})^2} 
\]
regardless of a relation between $\rho_{0}$ and $\rho_{1}$.
Since $y^{\mu}\to y:=\frac{\sqrt{\rho_0}u_0+\sqrt{\rho_1}u_1}
{\sqrt{\rho_0}+\sqrt{\rho_1}}$ as $\mu \to 0$, we have
\begin{equation}\label{Dmax}
\lim_{\mu\to 0}\sup\limits_{u_\delta\in(u_1,u_0)} 
\mathcal{D}=D_{\max}:=-\frac{1}{2}\frac{\rho_0
\rho_1(u_0-u_1)^3}{(\sqrt{\rho_0}+\sqrt{\rho_1})^2}.
\end{equation}

\noindent
{\it Case $B_3$} ($u_1< u_0< u_\delta$). Like in Case $B_1$, we have
\[\mathcal{D}=-\frac{1}{2}\rho_{1}\frac{\xi_{\delta}/\mu(u_{\delta}
-u_{1})^3}{(\sqrt{\xi_{\delta}/\mu}+\sqrt{\rho_{1}})^2}
<\frac{1}{2}\rho_{1}(u_0-u_{1})^3\]
for $t$ small enough.
\medskip

Obviously, if $u_0>u_1$, then 
\[
\frac{\rho_{0}\rho_{1}}{(\sqrt{\rho_{0}}
+\sqrt{\rho_{1}})^2}<\min\{\rho_{0},\rho_{1}\}
\text{ and } D_{\max}>-\frac{1}{2}\min\{\rho_{0},\rho_{1}\}(u_{0}-u_{1})^3.
\]
That is, $u_\delta=y^\mu$ is the value 
in (\ref{id_approx}) for which the corresponding 
solution satisfies the backward energy condition. 

\subsubsection{Limit of an energy admissible solution}
For (\ref{pgd}) we are able to easily obtain complete solution for $t>0$ 
using procedures from \cite{mn_sr2019}.
Thus, there exists a distributional limit of the constructed 
approximate solution satisfying the backward energy condition as $\mu \to 0$.
It will be taken to be the physically meaningful distributional solution 
to system (\ref{pgd}) with the distributional initial data (\ref{pgd_id}).
The energy dissipation analysis described in the details below
will show very interesting facts.
 
\begin{lemma}[Lemma 3.1.\ from \cite{mn_sr2019}] 
Let
\begin{equation*}
(\rho, u)(x,0)=\begin{cases}
(\rho_1,u_1), & x<0\\
(\rho_1,u_1), & x>0
\end{cases}+ (\xi_\delta,0)\delta,\; (\rho u)(x,0)=u_\delta \xi_\delta
\delta,
\end{equation*}
be the initial data for (\ref{pgd}). If
$u_0\geq u_\delta\geq u_1$, $\xi_\delta\geq 0$ and $\rho_0,\rho_1\geq 0$,
then the solution is the overcompressive shadow wave 
with strength $\xi(t)$ and speed $u_s(t)$ given by 
\begin{equation}\label{sol_xi}
\begin{split}
\xi(t) &=
\sqrt{\xi_\delta^2+\rho_0\rho_{1}[u]^2t^2+2\xi_\delta (u_\delta[\rho]-[\rho
u])t},\; \xi(0)=\xi_{\delta}\\
u_s(t) &=
\begin{cases}
\frac{1}{[\rho]}\Big([\rho u]+\frac{\rho_0\rho_{1}[u]^2t+\xi_\delta(u_\delta
[\rho]-[\rho u])}{\xi(t)}\Big), & \text{ if } \rho_0\neq \rho_1\\
\frac{\xi_\delta^2}{\xi^2(t)}(u_\delta-\frac{u_0+u_1}{2})+\frac{u_0+u_1}{2},
& \text{ if } \rho_0= \rho_1,
\end{cases} \; u_{s}(0)=u_{\delta}.
\end{split}
\end{equation}
The front of the resulting shadow wave
is $x=c(t)=\int_0^t u_s(\tau)\,d\tau$.
\end{lemma}

An interaction problem between two waves when at least one is a shadow one
can be interpreted as the initial value problem (\ref{pgd}, \ref{pgd_id}) 
with the initial data translated to the interaction point $(X,T)$, 
where \[\xi_\delta=\xi_l(T)+\xi_r(T),\;
u_\delta=\frac{u_{s,l}(T)\xi_l(T)+u_{s,r}(T)\xi_r(T)}{\xi_l(T)+\xi_r(T)}.\]
The values $\xi_i(T)$ and $u_{s,i}(T),$ $i=l,r$ denote strengths and speeds of
the left and right incoming wave at the interaction time $t=T$.
\medskip

For the readers convenience, we will present all possible solutions 
in the above cases. One can find some interesting things in their behaviour. 
\medskip 

\noindent
{\it Case} $A_1$.
The local solution consists of a shadow wave emanating from $x=-\mu/2$
and CD wave combination emanating from
$x=\mu/2$. The shadow wave is supported by $x=-\frac{\mu}{2}+y_{0}t$
with the speed 
$y_{0}:=\frac{u_0\sqrt{\rho_0}+u_\delta\sqrt{\frac{\xi_\delta}{\mu}}}
{\sqrt{\rho_0}+\sqrt{\frac{\xi_\delta}{\mu}}}$.
The speed $y_{0}$ is greater than the slope $u_\delta$ of the first
contact discontinuity, so two waves will interact at the time
$t=T_1=\frac{\mu}{y_{0}+\frac{\varepsilon}{2}-u_\delta}\sim \sqrt{\mu}.$
Note that we have used that the external shadow wave line given by
$x=-\frac{\mu}{2}+y_{0}t+\frac{\varepsilon}{2}t$ first intersects
the line $x=\frac{\mu}{2}+u_\delta t$, 
but the term $\frac{\varepsilon}{2}t$ can be neglected 
since $\varepsilon\ll \mu$ and $\rho_\varepsilon(t)\sim \varepsilon^{-1}$ 
(Lemma 3.2.\ from \cite{mn_sr2019}). 
The resulting shadow wave connects $U_{0}$ and vacuum state and propagates 
with the strength
$\xi(t) = \sqrt{\xi_{0}^2T_1^2+2\rho_0\xi_{0}T_1(u_0-y_{0})(t-T_1)}$,
where
$\xi_{0}:=\sqrt{\rho_0 \frac{\xi_\delta}{\mu}}(u_0-u_\delta)$,
and the speed $u_s(t)= u_0-\frac{\xi_{0}}{\xi(t)}(u_0-y_{0})$.
Both functions are increasing, $\xi(t)$ is non-negative and $u_s(t)\to u_0<u_1$
as $t\to \infty$. Thus, the resulting shadow wave will not interact with second
contact discontinuity. One can see that 
$y_{0}\to u_\delta$ and $\xi_{0}T_{1}\to \xi_{\delta}$ as $\mu \to 0$
by using the above expressions.
Straightforward calculation gives that 
\begin{equation}\label{delta_shockA}
(\rho,u)(x,t)=\begin{cases}
(\rho_0,u_0), & x<c(t)\\
\big(0,\frac{x}{t}\big), & c(t)<x<u_1 t\\
(\rho_1,u_1), & x>u_1 t
\end{cases}\;+ (\xi(t),0)\delta(x-c(t)) , 
\end{equation}
as the distributional limit of the approximate solution 
as $\mu \to 0$. The functions 
$\xi(t)$ and $u_s(t)$ are given in $(\ref{sol_xi})$, while
$c(t)=\int_{0}^t u_s(s)ds$.
\medskip

\noindent
{\it Case} $A_2$.
The solution consists of two
CD wave combinations, one emanating from $x=-\mu/2$ and
the other one from $x=\mu/2$. 
The approximate solution is given by
\begin{equation}\label{sol:a2}
U^\mu(x,t)=\begin{cases}
(\rho_0,u_0), & x<-\frac{\mu}{2}+u_0t\\
\big(0,\frac{x}{t}\big), & -\frac{\mu}{2}+u_0t<x<-\frac{\mu}{2}+u_\delta
t\\
\big(\frac{\xi_\delta}{\mu},u_\delta\big), & -\frac{\mu}{2}+u_\delta
t<x<\frac{\mu}{2}+u_\delta t\\
\big(0,\frac{x}{t}\big), & \frac{\mu}{2}+u_\delta t<x<\frac{\mu}{2}+u_1 t\\
(\rho_1,u_1), & x>\frac{\mu}{2}+u_1 t.
\end{cases}
\end{equation}
If $u_0<u_\delta<u_1,$ the distributional limit of (\ref{sol:a2}) is a
combination of two contact discontinuities and a delta shock supported by
line $x=u_\delta t$.
If $u_{\delta}$ coincides with  $u_{0}$ or $u_{1}$, then there is 
one instead of two contact discontinuities and the shadow wave has a constant
strength and characteristics speed. 
It is called the delta contact discontinuity (see \cite{mn_o2008}).
If $u_0=u_\delta=u_1$, the solution is a single delta contact
discontinuity.
\medskip

\noindent
{\it Case} $A_3$. 
The local solution consists of the CD wave combination 
emanating from $x=-\mu/2$ and the shadow wave from $x=\mu/2$. 
The second contact discontinuity in the combination (of the slope
$u_\delta$) interacts with the shadow wave
having the speed 
$y_{1}:=\frac{u_\delta\sqrt{\frac{\xi_\delta}{\mu}}
+u_1\sqrt{\rho_1}}{\sqrt{\frac{\xi_\delta}{\mu}}+\sqrt{\rho_1}}$.
The resulting shadow wave propagates with a speed $u_s(t)$ that increases
over time and satisfies $u_s(t)\to u_1>u_0$, $t\to \infty$.
That means that the left contact discontinuity will not
overtake the resulting shadow wave. The distributional limit of approximate
solution is a combination of contact discontinuity connecting $U_0$
and the vacuum state, and a weighted delta shock connecting 
the vacuum state to $U_1$, similarly to (\ref{delta_shockA}).
\medskip

\noindent
{\it Case} $B_1$.
Like in Case $A_1$, the approximate solution initially consists of 
the shadow wave emanating from $x=-\mu/2$ and the CD wave combination from 
$x=\mu/2$. The shadow wave interacts with the first contact discontinuity
at $t=T_1$ to form new shadow wave with speed $u_s(t)\to u_0$, $t \to \infty$.
Since $u_0>u_1$, at $t=T_2$ the shadow wave will overtake the
second contact discontinuity whose slope is $u_1$. The resulting one connects
$U_0$ to $U_1$. Its speed is increasing, so
$T_2<\frac{\mu}{y_{0}-u_1}\sim \sqrt{\mu}\to 0$ as $\mu\to 0.$ 
Thus, both interactions occur in the time $t\sim \sqrt{\mu}$ and the
distributional limit of approximate solution is
\begin{equation}\label{delta_shock}
U(x,t)=\begin{cases}
U_0, & x<c(t)\\
U_1, & x>c(t)
\end{cases}+(\xi(t),0)\delta(x-c(t)),
\end{equation}
where $\xi(t)$ and $u_s(t)$ are given by $(\ref{sol_xi})$
and $c(t)=\int_{0}^t u_s(s)ds$.
\medskip

\noindent
{\it Case $B_2$.}
The solution to (\ref{pgd}, \ref{id_approx})
consists of two overcompressive shadow waves interacting at time
$T=\frac{\mu}{y_{0}-y_{1}}\sim \sqrt{\mu}.$ One from $x=-\mu/2$
propagates with speed $y_{0}$ and strength $\xi_{0}t$, while the
other one propagates from $x=\mu/2$ with speed $y_{1}$ and strength $\xi_{1}t$.
The initial speed $u_s:=u_s(T+0)$ and strength $\xi_s:=\xi(T+0)$ of the
resulting overcompressive shadow wave are
$u_s=\frac{y_{0}\xi_{0}T+y_{1}\xi_{1}T}{\xi_{0}T+\xi_{1}T}$
and $\xi_s=\xi_{0}T+\xi_{1}T$.
We have $y_{0}-y_{1}\sim
\sqrt{\frac{\mu}{\xi_\delta}}\big(\sqrt{\rho_0}(u_0-u_\delta)
+\sqrt{\rho_1}(u_\delta-u_1)\big)$
as $\mu\to 0$. That fact together with $y_{0},y_{1}\to u_\delta$
as $\mu\to 0$ implies 
$\xi_s\to \xi_\delta$ and $u_s\to u_\delta$ as $\mu\to 0$.
A distributional limit as $\mu \to 0$ is an overcompressive delta shock
(\ref{delta_shock}) with $u_s(t)$
and $\xi(t)$ given by (\ref{sol_xi}).
\medskip

\noindent
{\it Case} $B_3$.
Immediately after the initial time, the solution consists 
of the CD wave combination and the shadow wave as in Case $A_3$.
The result of interaction between
the second contact discontinuity and the shadow wave is a new shadow wave
that propagates with a speed $u_s(t)\to u_1,$ $t\to \infty$. Now, $u_1<u_0$
and the first contact discontinuity will overtake the shadow wave when
$t=T_2\sim \sqrt{\mu}$. Both interactions occur when
$t\sim \sqrt{\mu}$, The distributional limit of 
approximate solution is (\ref{delta_shock}).
\medskip

Let us now write down solutions satisfying the backward energy condition. 
If $u_0\leq u_1$, the weak solution that solves (\ref{pgd},
\ref{pgd_id}) and satisfies backward energy condition is 
one or two (when $\mathop{\rm sign}(u_{0}u_{1})<0$) 
contact discontinuities combined with the delta contact 
discontinuity given by
\begin{equation}\label{sol_A}
U(x,t)=\begin{cases}
(\rho_0,u_0), & x<u_0t\\
\big(0,\frac{x}{t}\big), & u_0t<x<u_\delta t\\
\big(0,\frac{x}{t}\big), & u_\delta t<x<u_1 t\\
(\rho_1,u_1), & x>u_1 t,
\end{cases}+(\xi_\delta,0)\delta(x-u_\delta t)\end{equation}
where $u_\delta\in[u_0,u_1]$ is given in (\ref{iemin}) so that the total 
initial energy is minimized.

If $u_0>u_1$, the weak solution that satisfies the backward energy condition is
given by
\begin{equation}\label{sol_B}
U(x,t)=\begin{cases}
(\rho_0,u_0), & x<y t\\
(\rho_1,u_1), & x>y t
\end{cases}+ (\xi_{\delta}+\xi t,0)\delta(x-yt),
\end{equation}
where $\xi=\sqrt{\rho_0 \rho_1}(u_0-u_1),$
$y=\frac{u_0\sqrt{\rho_{0}}+u_{1}\sqrt{\rho_{1}}}{\sqrt{\rho_{0}}
+\sqrt{\rho_{1}}}$.
The speed and strength of
(\ref{sol_B}) are obtained from (\ref{sol_xi}) by putting $u_\delta=y$.
\medskip

The above results are summarized in the following theorem.

\begin{theorem}
Let (\ref{pgd}, \ref{pgd_id}) be given, where $\rho_0,\rho_1>0$ and
$\xi_\delta>0$. If $u_0>u_1$, the admissible weak solution satisfying the
backward energy condition is the delta shock wave propagating
with constant speed $y$ given by (\ref{sol_B}).
If $u_0\leq u_1$, the backward admissible solution locally conserves energy in
the sense of distributions and it is given by (\ref{sol_A}).
\end{theorem}

\begin{remark}
Additionally, if the initial speed of a shadow wave (\ref{wSDW_l,r}) equals
$y$, its local energy production is maximal and equals $D_{\max}$
defined in (\ref{Dmax}) for each $t$. 
That is a consequence of the fact that
for a shadow wave with the initial speed $y$ we have $u_{s}(t)=y,$ $t\geq 0$. 
We say that such a wave is in its equilibrium state.
Since $u_{s}(t)\to y$ as $t\to \infty$ the following holds: 
A (shadow) wave that starts in the non-equilibrium state 
(a nonpositive local energy production across
its front has not reached a maximum value) will adjust its speed
to approach the equilibrium state. That is not the case in general, as it will
be demonstrated for systems (\ref{chap}) and (\ref{genchap}).
\end{remark}

\subsection{Chaplygin model}

The backward energy condition can also be applied successfully for 
the Chaplygin model (\ref{chap}). The energy density is 
$\eta=\rho u^2+\frac{1}{\rho}$ with the 
flux $Q=\rho u^3-\frac{u}{\rho}$.
Using (\ref{entr}) we get that the local energy production of a
shadow wave connecting $(\rho_0,u_0)$ to $(\rho_1,u_1)$ equals
\[\mathcal{D}=-s[\eta]+[Q]+\kappa s^2,\]
where $\kappa$ and $s$ are defined in (\ref{speed_chap}).
The structure of approximate solution and its analysis are similar to the
one for pressureless gas dynamics system. 

\begin{theorem}
Let the problem (\ref{chap}, \ref{pgd_id}) be given, 
where $\rho_0,\rho_1>0$ and $\xi_\delta>0$. 
If $\lambda_1(\rho_0,u_0)> \lambda_2(\rho_1,u_1)$, the
admissible weak solution that meets the backward energy condition
is a weighted delta shock.
If $\lambda_1(\rho_0,u_0)\leq \lambda_2(\rho_1,u_1)$, the solution satisfying
the backward energy condition locally conserves energy in the sense of
distributions and 
\[u_\delta=\begin{cases}
\lambda_2(\rho_{1}, u_{1}), & \text{if } \lambda_2(\rho_{1}, u_{1})<0\\
\lambda_1(\rho_{0}, u_{0}), & \text{if } \lambda_1(\rho_{0}, u_{0})>0\\
0, & \text{otherwise}.
\end{cases}\]
\end{theorem}
\begin{proof}
Let us first consider the case 
$\lambda_1(\rho_0,u_0)\leq \lambda_2(\rho_1,u_1)$.
The situation is analogous to the one obtained for 
$u_{0}\leq u_{1}$ in the pressureless system. 

If $u_\delta\in [\lambda_1(\rho_0,u_0), \lambda_2(\rho_1,u_1)]$, 
the solution to the problem (\ref{chap}, \ref{id_approx}) 
is a combination of contact discontinuities, 
and its local energy production equals zero for each
$u_\delta$. Otherwise it would be negative due to presence of a shadow wave.
The distributional limit of such solution is 
\begin{equation*}
U(x,t)=\begin{cases}
(\rho_0,u_0), & x<\lambda_1(\rho_{0},u_{0})t\\
(\rho_{m_1},u_{m_1}), & \lambda_1(\rho_{0},u_{0})t<x<u_\delta t\\
(\rho_{m_2},u_{m_2}), & u_\delta t<x<\lambda_2(\rho_1,u_1) t\\
(\rho_1,u_1), & x>\lambda_2(\rho_1,u_1) t
\end{cases}+(\xi_\delta,0)\delta(x-u_\delta t),
\end{equation*}
where
\[
\begin{split}
(\rho_{m_1},u_{m_1}) &=\Big(\frac{2}{u_\delta-u_0+\rho_0^{-1}},
\frac{u_\delta+u_0-\rho_{0}^{-1}}{2}\Big),\\
(\rho_{m_2},u_{m_2}) &=\Big(\frac{2}{u_1+\rho_{1}^{-1}-u_\delta},
\frac{u_\delta+u_1+\rho_{1}^{-1}}{2}\Big).
\end{split}
\]
The value of $u_\delta$ that minimizes $u_\delta^2$, and consequently
the total initial energy is a proper choice for $u$ component
of initial data.
 
There are three possible types of solution when $\lambda_1(\rho_0,u_0)>
\lambda_2(\rho_1,u_1)$.
Let $u_{\delta}<\lambda_{1}(\rho_{0},u_{0})$. Then $U_{0}=(\rho_{0},u_{0})$ and
$U_\delta=\big(\frac{\xi}{\mu},u_{\delta}\big)$ are connected by 
the shadow wave with strength
$\kappa=\sqrt{[\rho u^2]_{0}-[\rho]_{0}[\rho u^2-\rho^{-1}]_{0}}$ and speed
$s=\frac{[\rho u]_{0}}{[\rho]_{0}}+\frac{\kappa}{[\rho]_{0}}$. 
We have used the notation $[\cdot]_{0}=\cdot|_{U_{\delta}}-\cdot|_{U_{0}}$, 
see (\ref{speed_chap}). Then
\[
\begin{split}
s-u_\delta & =(u_\delta-u_0)\Big(\frac{\rho_{0}}{\xi_{\delta}/\mu-\rho_{0}}
+\frac{\kappa}{(\xi_{\delta}/\mu-\rho_{0})(u_{\delta}-u_{0})}\Big)=:
(u_\delta-u_0)d_{1}^{\mu} \text{ and}\\
s-u_0 & =(u_\delta-u_0)\Big(\frac{\xi_{\delta}/\mu}
{\xi_{\delta}/\mu-\rho_{0}}+\frac{\kappa}
{(\xi_{\delta}/\mu-\rho_{0})(u_{\delta}-u_{0})}\Big)
=:(u_\delta-u_0)d_{2}^{\mu}.
 \end{split}
\]
The local energy production of the shadow wave connecting 
$U_{0}$ to $U_{\delta}$ equals
\[
\begin{split}
 \mathcal{D} &=\frac{\xi_\delta}{\mu}(u_\delta-s)^3+\rho_{0}(s-u_0)^3
+\rho_{0}^{-1}(u_0-s)+\frac{\mu}{\xi_{\delta}}(s-u_{0})\\
&= (u_{0}-u_{\delta})^3\Big((d_{1}^{\mu})^3\frac{\xi_{\delta}}{\mu}
-(d_{2}^{\mu})^3\rho_{0}\Big)+(u_{0}-u_{\delta})\Big(\rho_{0}^{-1}d_{2}^{\mu}
-\frac{\mu}{\xi_{\delta}}d_{1}^{\mu}\Big).
\end{split}
\] 
We have $d_{1}^{\mu}=\mathcal{O}(\sqrt{\mu})$ and
$d_{2}^{\mu}=1+\mathcal{O}(\sqrt{\mu}),$ $\mu\to 0.$ Thus,
\[\begin{split}
&\mathcal{D}\approx\rho_0(u_\delta-u_0)^3+\frac{1}{\rho_0}(u_0-u_\delta)
\text{ for } \mu \text{ small enough}.
\end{split}\]
Similarly, if $u_\delta>\lambda_2(U_1)$, then 
$U_\delta$ and $U_1=(\rho_{1},u_{1})$
are connected by a shadow wave and 
\[\begin{split}
& \mathcal{D}\approx\rho_1(u_1-u_\delta)^3+\frac{1}{\rho_1}(u_\delta-u_1)
\text{ for } \mu \text{ small enough}.
\end{split}\]

If $u_\delta\leq \lambda_2(U_{1})$ one can see that $\mathcal{D}$
has its maximum at $u_\delta=\lambda_2(U_{1})$. In the case
$u_\delta\geq \lambda_1(U_{0})$, $\mathcal{D}$ has it at
$u_\delta=\lambda_1(U_{0})$ as it was the case for 
the pressureless gas dynamics.
In those two cases solutions to (\ref{chap}, \ref{id_approx}) are 
combinations of two contact discontinuities and shadow waves.
The only case left to consider is when 
$\lambda_1(U_{0})\geq u_\delta\geq \lambda_2(U_{1})$. The solution is
the combination of two overcompressive shadow waves which
interact at time $T\sim \sqrt{\mu}$ giving a single overcompressive shadow wave
with variable speed (see \cite{mn2014} for the proof).
In that case
\[
\mathcal{D}\approx \rho_0(u_\delta-u_0)^3+\frac{1}{\rho_0}
(u_0-u_\delta)+ \rho_1(u_1-u_\delta)^3+\frac{1}{\rho_1}(u_\delta-u_1),
\, \mu\to 0.
\]
The local energy production of such solution is maximized for 
\[
u_\delta=\begin{cases}
\lambda_2(\rho_{1}, u_{1}), 
& \text{if } x_{\ast}\leq \lambda_2(\rho_{1}, u_{1})\\
x_{\ast}, & \text{if } x_{\ast}\in 
(\lambda_2(\rho_{1}, u_{1}),\lambda_1(\rho_{0}, u_{0}))\\
\lambda_1(\rho_{0}, u_{0}), & \text{if } x_{\ast}
\geq \lambda_1(\rho_{0}, u_{0}),
\end{cases}
\] 
where $x_{\ast}=\frac{[\rho u]}{[\rho]}+\frac{1}{[\rho]}\sqrt{\rho_0
\rho_1[u^2]+\frac{1}{3}[\rho][\rho^{-1}]}$. 
It can be proved that 

$x_{\ast}\leq \lambda_{1}(\rho_{0}, u_{0})$ if $[u]+\rho_{0}^{-1}\leq
-\rho_{1}^{-\frac{1}{2}}\sqrt{\rho_{0}^{-1}+\frac{1}{3}[\rho^{-1}]}$ and

$x_{\ast}\geq \lambda_2(\rho_{1}, u_{1})$ if $[u]+\rho_{1}^{-1}\leq
-\rho_{0}^{-\frac{1}{2}}\sqrt{\rho_{1}^{-1}+\frac{1}{3}[\rho^{-1}]}$. 

Thus, in this case the backward energy condition is satisfied for 
some $u_{\delta}$ between $\lambda_1(U_{0})$ and $\lambda_2(U_{1})$. 

Unlike the system (\ref{pgd}), the distributional limit of this solution is a
weighted delta shock and its local energy production is not constant with
respect to $t$ which follows from the fact that the speed of shadow wave
joining $U_0$ and $U_1$ is not equal to $x_{\ast}.$ 
Such a wave will not reach its equilibrium state in infinity.
\end{proof}

\subsection{Generalized Chaplygin gas model}

The last model analysed in the paper is system
(\ref{genchap}) with energy pair (\ref{en_pair}).
Denote the following values
\[
A_1=u_{0}-\rho_{0}^{-\frac{1+\alpha}{2}}, \; 
A_2=u_{0}+\frac{2\sqrt{\alpha}}{1+\alpha}\rho_{0}^{-\frac{1+\alpha}{2}},\; 
B_1=u_{1}-\frac{2\sqrt{\alpha}}{1+\alpha}\rho_{1}^{-\frac{1+\alpha}{2}}, \; 
B_2=u_{1}+\rho_{1}^{-\frac{1+\alpha}{2}}.
\]
Depending on relations between the above constants,
there are the following possibilities for a small time solution.
\begin{itemize} 
\item The state $(\rho_0,u_0)$ can be connected to
$(\frac{\xi_\delta}{\mu},u_\delta)$ by
\begin{itemize}
\item $R_1+R_2$ if $u_{\delta}>A_2$
\item $S_1+R_2$ if $A_1<u_{\delta}\leq A_2$ 
\item shadow wave if $u_\delta\leq A_1$
\end{itemize}
\item The state $(\frac{\xi_\delta}{\mu},u_\delta)$ can be connected to
$(\rho_1,u_1)$ by
\begin{itemize}
\item $R_1+R_2$ if $u_{\delta}<B_1$
\item $R_1+S_2$ if $B_1\leq u_{\delta}< B_2$ 
\item shadow wave if $u_\delta\geq B_2$. 
\end{itemize}
\end{itemize}

Note that the local energy production for a rarefaction wave 
equals zero due to their continuity and it is negative for
other waves. The following lemmas will describe the production 
in all the above cases.

\begin{lemma} 
The local energy production of the shadow wave connecting $(\rho_0,u_0)$ and
$\big(\frac{\xi_\delta}{\mu}, u_\delta\big)$, $u_{\delta}\leq A_{1}$ equals
\[
\mathcal{D}=f_{1}(u_{\delta}):=
\frac{1}{2}\rho_0(u_\delta-u_0)^3-\frac{\alpha}{1+\alpha}
\rho_{0}^{-\alpha}(u_\delta-u_0)+\mathcal{O}(\sqrt{\mu})
\text{ as } \mu \to 0
\]
and the limit has a maximum at $A_{1}$.

For the shadow wave connecting $\big(\frac{\xi_\delta}{\mu}$,$u_\delta\big)$ 
and $(\rho_{1},u_{1})$, $u_{\delta}\geq B_{2}$ we have
\[
\mathcal{D}=f_{2}(u_{\delta}):=\frac{1}{2}\rho_1(u_1-u_\delta)^3
-\frac{\alpha}{1+\alpha}\rho_{1}^{-\alpha}(u_1-u_\delta)
+\mathcal{O}(\sqrt{\mu})
\text{ as } \mu \to 0.
\]
The limit has a maximum at $B_2$.

The solution composed of two shadow waves exists if $B_2\leq u_{\delta}\leq
A_1$ and its local energy production has a maximum at
\begin{equation*}\ 
u_\delta=\frac{[\rho u]}{[\rho]}+\frac{1}{[\rho]}\sqrt{\rho_0 \rho_1 [u]^2
+\frac{2}{3}\frac{\alpha}{1+\alpha}[\rho][\rho^{-\alpha}]}=:x_0
\end{equation*}
if $x_0\in [B_2,A_1]$. Otherwise, a maximum point will be $B_2$ or $A_1$.
\end{lemma}

\begin{proof}
By using the cubic function properties, we can see that $f_{1}$ has a 
maximum at $u_\delta=u_0-\sqrt{\alpha}\sqrt{\frac{2}{3(1+\alpha)}}
\rho_0^{-\frac{1+\alpha}{2}}>A_1$, while $f_{2}$ has a maximum at
$u_\delta=u_1+\sqrt{\alpha}\sqrt{\frac{2}{3(1+\alpha)}}
\rho_1^{-\frac{1+\alpha}{2}}<B_2$.
The sum $f_{1}+f_{2}$ 
has a maximum at 
\[ 
x_{0}=\frac{[\rho u]}{[\rho]}+\frac{1}{[\rho]}
\sqrt{\rho_0 \rho_1 [u]^2+\frac{2}{3}\frac{\alpha}{1+\alpha}
[\rho][\rho^{-\alpha}]}.
\]
If $x_0\in (B_2,A_1)$, then $u_{\delta}=x_{0}$. 
Otherwise, $u_{\delta}$ is exactly one of the endpoints of the interval, 
$u_{\delta}=B_2$ or $u_{\delta}=A_1$ depending on the given values 
for $(u_{0},\rho_{0})$ and $(u_{1},\rho_{1})$, since a solution composed of two shadow waves exists if 
$B_2\leq u_{\delta}\leq A_1$.
\end{proof}

\begin{lemma}
The local energy production of the 
$S_1+R_2$ solution connecting $(\rho_0,u_0)$ and
$\big(\frac{\xi_\delta}{\mu}, u_\delta\big)$ is
\[
\mathcal{D}= A(\rho_{0},\rho_{m})\Big(\frac{1}{2}
\frac{\rho_{m}-\rho_{0}}{\rho_{m}}(\rho_{0}^{-\alpha}
-\rho_{m}^{-\alpha})+\frac{1}{1+\alpha}(\rho_{m}^{-\alpha}
-\rho_{0}^{-\alpha})+\frac{\alpha}{1+\alpha}\rho_{m}^{-\alpha}
\frac{\rho_{m}-\rho_{0}}{\rho_{m}}\Big).
\]
The value $A(\rho_{0},\rho_{m})$ is defined in (\ref{A_0m}), while $\rho_m$
satisfies 
\[
u_0-\frac{\rho_m-\rho_0}{\rho_m}A(\rho_{0},\rho_{m})
=u_\delta-\frac{2\sqrt{\alpha}}{1+\alpha}\rho_{m}^{-\frac{1+\alpha}{2}},\; 
\rho_{m}>\rho_0.
\] 
Also, $\mathcal{D}$ increases with respect to $u_{\delta}$ and $0\leq 
\frac{\partial \mathcal{D}}{\partial u_\delta}
\leq m_1:=\big(\frac{3}{2}-\frac{\alpha}{1+\alpha}\big)\rho_{0}^{-\alpha}$ 
for $u_\delta\in [A_1,A_2]$.
\end{lemma}
\begin{proof}
Put
\[
\begin{split}
a &:=\frac{1}{2}\frac{\rho_{m}-\rho_{0}}{\rho_{m}}
(\rho_{0}^{-\alpha}-\rho_{m}^{-\alpha}) \\
b &:=\frac{1}{1+\alpha}(\rho_{0}^{-\alpha}
-\rho_{m}^{-\alpha})-\frac{\alpha}{1+\alpha}\rho_{m}^{-\alpha}
\frac{\rho_{m}-\rho_{0}}{\rho_m}.
\end{split}
\]
Then $\mathcal{D}= A(\rho_{0},\rho_{m})(a-b)$. The function $a$ increases with
$\rho_{m}$ and $a=0$ for $\rho_m=\rho_0$, i.e.\ $a>0.$ 
Also, $b$ increases with $\rho_{m}$, $b=0$ if $\rho_{m}=\rho_{1}$, so $b>0.$
Using The Chain Rule and the Implicit Function theorem we obtain
\[
\frac{\partial \mathcal{D}}{\partial u_\delta}
=\frac{\partial \mathcal{D}}{\partial \rho_m}
\frac{\partial \rho_{m}}{\partial u_\delta},
\]
where
\[
\begin{split}
\frac{\partial \mathcal{D}}{\partial \rho_{m}} 
&=-\frac{1}{2}A(\rho_{0},\rho_{m})\big(2a-(1+\alpha)b\big)
\frac{a+b}{(\rho_{m}-\rho_{0})(\rho_{0}^{-\alpha}-\rho_{m}^{-\alpha})}\\
\frac{\partial \rho_{m}}{\partial u_\delta} 
& =-\frac{\rho_{m}(\rho_{0}^{-\alpha}
-\rho_{m}^{-\alpha})}{(\rho_{0}^{-\alpha}
-\rho_{m}^{-\alpha})(\sqrt{\alpha}\rho_{m}^{-\frac{1+\alpha}{2}}
+\frac{\rho_{0}}{\rho_{m}}A(\rho_{0},\rho_{m}))
+\frac{1}{2}A(\rho_{0},\rho_{m})\big(2a-(1+\alpha)b\big)}.
\end{split}
\]
The inequality $\frac{\partial \mathcal{D}}{\partial u_\delta}\geq 0$ follows
from $\rho_{m}\geq \rho_{0}$ and the fact that
$\rho_{m}\big(2a-(1+\alpha)b\big)$ increases with $\rho_m$ increases, 
so $2a-(1+\alpha)b\geq 0.$
Next, we have
\[
\frac{\partial \mathcal{D}}{\partial u_\delta}
\leq \frac{\rho_{m}}{\rho_{m}-\rho_{0}}(a+b)\leq m_1.
\]
The above inequality holds because the values
$\frac{\rho_{m}}{\rho_{m}-\rho_{0}}a$ and $\frac{\rho_{m}}{\rho_{m}-\rho_{0}}b$
are nonnegative, increasing with $\rho_{m}$
and 
\[
\lim_{\rho_{m}\to \infty}\frac{\rho_{m}}{\rho_{m}-\rho_{0}}a
=\frac{1}{2}\rho_0^{-\alpha},\; 
\lim_{\rho_{m}\to \infty}\frac{\rho_{m}}{\rho_{m}-\rho_{0}}b
=\frac{1}{1+\alpha}\rho_0^{-\alpha}.
\]
\end{proof}

The following lemma can be proved in the same way.
\begin{lemma} 
The local energy production of the $R_1+S_2$ solution connecting
$\big(\frac{\xi_\delta}{\mu}, u_\delta\big)$ and $(\rho_1,u_1)$ is
\[
\mathcal{D}= A(\rho_{1},\rho_{m})\Big(\frac{1}{2}
\frac{\rho_{m}-\rho_{1}}{\rho_{m}}(\rho_{1}^{-\alpha}
-\rho_{m}^{-\alpha})+\frac{1}{1+\alpha}(\rho_{m}^{-\alpha}
-\rho_{1}^{-\alpha})+\frac{\alpha}{1+\alpha}\rho_{m}^{-\alpha}
\frac{\rho_{m}-\rho_{1}}{\rho_{m}}\Big),
\]
where $\rho_m$ satisfies 
\[
u_1+\frac{\rho_m-\rho_1}{\rho_m}A(\rho_{1},\rho_{m})
=u_\delta+\frac{2\sqrt{\alpha}}{1+\alpha}\rho_{m}^{-\frac{1+\alpha}{2}}, \; 
\rho_{m}>\rho_1.
\] 
Also, $\mathcal{D}$ decreases with respect to $u_{\delta}$ and
$-m_2:=-\big(\frac{3}{2}-\frac{\alpha}{1+\alpha}\big)\rho_{1}^{-\alpha}\leq
\frac{\partial \mathcal{D}}{\partial u_\delta}\leq 0$ for $u_\delta\in
[B_1,B_2]$.
\end{lemma}


All the above lemas have to be used in the proof of tollowing theorem.
\begin{theorem}
There exists unique $u_\delta \in [\min\{A_2,B_1\}, \max\{A_2,B_1\}]$ and the
corresponding solution to the problem (\ref{genchap}, \ref{pgd_id}) satisfying
backward energy condition.
\end{theorem}

The proof is based on an analysis of all possible relations between
the values $A_{1}$, $A_{2}$, $B_{1}$ and $B_{2}$ where 
one uses the above lemmas and direct calculations. There are a lot of technical
details, so we will present the table with results in the Appendix.

\begin{remark}
To obtain the form of approximate solution for each $t>0$ and a
corresponding distributional limit, it is necessary to know the result of all
interaction problems including at least one shadow wave. That also includes
interactions bet\-ween shadow and rarefaction waves. That is left for future
work.
\end{remark}

\section*{Appendix}

\hspace*{-.50cm}
\includegraphics[scale=0.27]{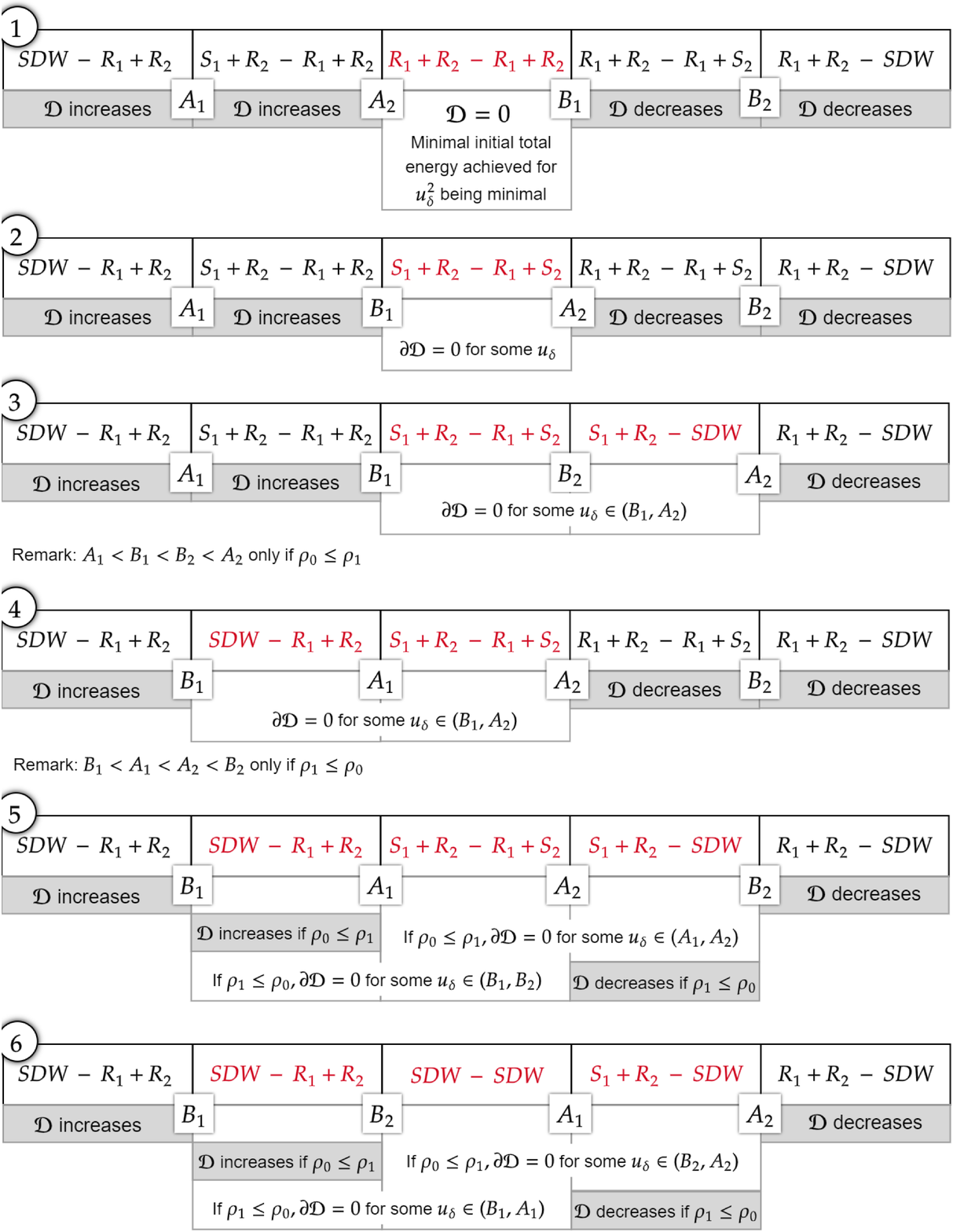}

\end{document}